\documentclass[a4paper,pdftex,reqno,10pt]{amsart}

\allowdisplaybreaks[1]

\usepackage{geometry}
\usepackage{graphicx}
\usepackage{enumerate}
\usepackage{courier}
\usepackage{times}
\usepackage{amssymb}
\usepackage{amsmath}

\usepackage{hyperref}

\theoremstyle{plain}
\newtheorem{Theorem}{Theorem}[section]

\newtheorem{Corollary}[Theorem]{Corollary}
\newtheorem{Lemma}[Theorem]{Lemma}

\newenvironment{Proof}
{\begin{trivlist}\item[]{{\sc Proof.}}}{\hfill{$\square$}\noindent\end{trivlist}}

\newcommand{\gaussm}[3]{\genfrac{[}{]}{0pt}{}{#1}{#2}_{#3}}

\newcommand{\F}[2]{\mathbb{F}_{#2}^{#1}}

\theoremstyle{definition}
\newtheorem{Definition}[Theorem]{Definition}

\theoremstyle{remark}

\begin{document}


\title{Upper bounds for partial spreads}

\author[Sascha Kurz]{Sascha Kurz$^\star$}
\address{Department of Mathematics, University of Bayreuth, 95440 Bayreuth, Germany}
\email{sascha.kurz@uni-bayreuth.de}
\thanks{$^\star$ The work of the author was supported by the ICT COST Action IC1104
and grant KU 2430/3-1 -- Integer Linear Programming Models for Subspace Codes and Finite Geometry
from the German Research Foundation.}

\date{}

\begin{abstract}
  A \emph{partial $t$-spread} in $\F{n}{q}$ is a collection of $t$-dimensional 
  subspaces with trivial intersection such that each non-zero vector is covered at most once. 
  We present some improved upper bounds on the maximum sizes.
  
  \medskip
  
\end{abstract}

\maketitle

\section{Introduction}

\noindent
Let $q>1$ be a prime power and $n$ a positive integer. A \emph{vector space partition} $\mathcal{P}$ of $\F{n}{q}$ is a collection of 
subspaces with the property that every non-zero vector is contained in a unique member of $\mathcal{P}$. If $\mathcal{P}$  contains 
$m_d$ subspaces of dimension $d$, then $\mathcal{P}$ is of type $k^{m_k}\dots 1^{m_1}$. We may leave out some of the cases with $m_d=0$. 
Subspaces of dimension~$1$ are called \emph{holes}. If there is at least one non-hole, then $\mathcal{P}$ is called non-trivial. 

A \emph{partial $t$-spread} in $\F{n}{q}$ is a collection of $t$-dimensional subspaces such that the non-zero vectors are covered at most 
once, i.e., a vector space partition of type $t^{m_t}1^{m_1}$. By $A_q(n,2t;t)$ we denote the maximum value of $m_t$ 
\footnote{The more general notation $A_q(n,2t-2w;t)$ denotes the maximum cardinality of a collection of $t$-dimensional subspaces, 
whose pairwise intersections have a dimension of at most $w$. Those objects are called \emph{constant dimension codes}, see 
e.g.~\cite{etzionsurvey}. For known bounds, we refer to \url{http://subspacecodes.uni-bayreuth.de} \cite{TableSubspacecodes} containing 
also the generalization to \emph{subspace codes} of mixed dimension.}.  
Writing $n=kt+r$, with $k,r\in\mathbb{N}_0$ and $r\le t-1$, we can state that for $r\le 1$ or $n\le 2t$ the 
exact value of $A_q(n,2t;t)$ was known for more than forty years \cite{beutelspacher1975partial}. Via a computer search the cases 
$A_2(3k+2,6;3)$ were settled in 2010 by El-Zanati et al.~\cite{spreadsk3}. In 2015 the case $q=r=2$ was resolved by continuing the original approach 
of Beutelspacher \cite{kurzspreads}, i.e., by \textit{considering} the set of holes in $(n-2)$-dimensional subspaces and some averaging 
arguments. Very recently, N{\u{a}}stase and Sissokho found a very clear generalized averaging method for the number of holes in 
$(n-j)$-dimensional subspaces, where $j\le t-2$, and general $q$, see \cite{nastase2016maximum}. Their Theorem~5 determines the exact 
values of $A_q(kt+r,2t;t)$ in all cases where $t>\gaussm{r}{1}{q}:=\frac{q^r-1}{q-1}$.
Here, we streamline and generalize their approach leading to improved upper bounds on $A_q(n,2t;t)$, c.f.~\cite{nastase2016maximumII}.  

\section{Subspaces with the minimum number of holes}

\begin{Definition}
  A vector space partition $\mathcal{P}$ of $\F{n}{q}$ has \emph{hole-type} $(t,s,m_1)$, if it is of type $t^{m_t}\dots s^{m_s}1^{m_1}$, 
  for some integers $n>t\ge s\ge 2$, $m_i\in\mathbb{N}_0$ for $i\in\{1,s,\dots,t\}$, and $\mathcal{P}$ is non-trivial.
\end{Definition}

\begin{Lemma} (C.f.~\cite[Proof of Lemma 9]{nastase2016maximum}.)
  \label{lemma_modulo}
  Let $\mathcal{P}$ be a non-trivial vector space partition of $\F{n}{q}$ of hole-type $(t,s,m_1)$ and $l,x\in\mathbb{N}_0$ with $\sum_{i=s}^t m_i =lq^s+x$. 
  $\mathcal{P}_{H}=\{U\cap H\,:\, U\in\mathcal{P}\}$ is a vector space partition of type $t^{m_t'}\dots (s-1)^{m_{s-1}'}1^{m_1'}$, 
  for a hyperplane $H$ with $\widehat{m}_1$ holes (of $\mathcal{P}$). We have $\widehat{m}_1\equiv \frac{m_1+x-1}{q}\pmod{q^{s-1}}$. 
  If $s>2$, then $\mathcal{P}_{H}$ is non-trivial and $m_1'=\widehat{m}_1$.
\end{Lemma}
\begin{Proof} 
  If $U\in\mathcal{P}$, then $\dim(U)-\dim(U\cap H)\in\{0,1\}$ for an arbitrary hyperplane $H$. 
  Since $\mathcal{P}$ is non-trivial, we have $n\ge s$. For $s>2$, 
  counting the $1$-dimensional subspaces of $\F{n}{q}$ and $H$, via $\mathcal{P}$ and $\mathcal{P}_H$, yields
  $$
    \left(lq^s+x\right)\cdot\gaussm{s}{1}{q}+aq^s+m_1=\gaussm{n}{1}{q}
    \quad\text{and}\quad
    \left(lq^s+x\right)\cdot\gaussm{s-1}{1}{q}+a'q^{s-1}+\widehat{m}_1=\gaussm{n-1}{1}{q}
  $$
  for some $a,a'\in\mathbb{N}_0$. Since $1+q\cdot\gaussm{n-1}{1}{q}-\gaussm{n}{1}{q}=0$ we conclude
  $1+q\widehat{m}_1-m_1-x\equiv 0\pmod {q^s}$. Thus, $\mathbb{Z}\ni \widehat{m}_1\equiv \frac{m_1+x-1}{q} \pmod {q^{s-1}}$.
  For $s=2$ we have
  $$
    \left(lq^2+x\right)\cdot(q+1)+aq^2+m_1=\gaussm{n}{1}{q}
    \quad\text{and}\quad
    \left(lq^2+x\right)\cdot 1+a'q+\widehat{m}_1=\gaussm{n-1}{1}{q}
  $$ 
  leading to the same conclusion $\widehat{m}_1\equiv \frac{m_1+x-1}{q} \pmod {q^{s-1}}$.   
\end{Proof}
   
\begin{Lemma} (C.f.~\cite[Proof of Lemma 9]{nastase2016maximum}.)
  \label{key_lemma}
  Let $\mathcal{P}$ be a vector space partition of $\F{n}{q}$ of hole-type $(t,s,m_1)$, $l,x\in\mathbb{N}_0$ with $\sum_{i=s}^t m_i =lq^s+x$, 
  and $b,c\in\mathbb{Z}$ with $m_1=bq^s+c\ge 1$. If $x\ge 1$, 
  then there exists a hyperplane $\widehat{H}$ with $\widehat{m}_1=\widehat{b}q^{s-1}+\widehat{c}$ holes, where 
  $\widehat{c}:=\frac{c+x-1}{q}\in\mathbb{Z}$ and $b>\widehat{b}\in\mathbb{Z}$.
\end{Lemma}
\begin{Proof} 
  Apply Lemma~\ref{lemma_modulo} and observe $m_1\equiv c\pmod {q^s}$. 
  Let the number of holes in $\widehat{H}$ be minimal. Then, 
  \begin{equation}
    \label{ie_m1}
    \widehat{m}_1\le \text{average number of holes per hyperplane}=m_1\cdot\gaussm{n-1}{1}{q}/\gaussm{n}{1}{q}<\frac{m_1}{q}.
  \end{equation}
  Assuming $\widehat{b}\ge b$ yields $q\widehat{m}_1\ge q\cdot(bq^{s-1}+\widehat{c})=bq^s+c+x-1\ge m_1$, which contradicts Inequality~(\ref{ie_m1}). 
\end{Proof}

\begin{Corollary}
  \label{cor_j_times}
  Using the notation from Lemma~\ref{key_lemma}, let $\mathcal{P}$ be a non-trivial vector space partition with $x\ge 1$ 
  and $f$ be the largest integer such that $q^f$ divides $c$. For each $0\le j\le s-\max\{1,f\}$ there exists an 
  $(n-j)$-dimensional subspace $U$ containing $\widehat{m}_1$ holes  
  with $\widehat{m}_1\equiv \widehat{c}~\pmod {q^{s-j}}$ and $\widehat{m}_1\le (b-j)\cdot q^{s-j}+\widehat{c}$, where
  $\widehat{c}=\frac{c+\gaussm{j}{1}{q}\cdot (x-1)}{q^j}$.
\end{Corollary}
\begin{proof}
  Observe $\widehat{m}_1\equiv c\not\equiv 0\pmod{q^{s-j}}$, i.e., $\widehat{m}_1\ge 1$, for all $j<s-f$.
\end{proof}

\begin{Lemma}
  \label{lemma_application_partial_spreads}
  Let $\mathcal{P}$ be a non-trivial vector space partition of type $t^{m_t}1^{m_1}$ of $\F{n}{q}$ with $m_t=lq^t+x$, 
  where $l=\frac{q^{n-t}-q^r}{q^t-1}$, $x\ge 2$, $t=\gaussm{r}{1}{q}+1-z+u>r$, $q^f | x-1$, $q^{f+1} \!\nmid\! x-1$, and  
  $f,u,z,r,x\in\mathbb{N}_0$. For $\max\{1,f\}\le y\le t$ there exists a $(n\!-\!t\!+\!y)$-dimensional 
  subspace $U$ with $L\le (z\!+\!y\!-\!1)q^y\!+\!w$ holes, where $w\!=\!-(x\!-\!1)\gaussm{y}{1}{q}$ and $L\!\equiv w\! \pmod {q^y}$. 
\end{Lemma}
\begin{Proof}
  Apply Corollary~\ref{cor_j_times} with $s=t$, $j=t-y$, $b=\gaussm{r}{1}{q}$, and 
  $m_1
  =\gaussm{r}{1}{q}q^t-\gaussm{t}{1}{q}(x-1)$. 
\end{Proof}

\begin{Lemma}
  \label{lemma_standard_equations_reduced}
  Let $\mathcal{P}$ be a vector space partition of $\F{n}{q}$ with $c\ge 1$ holes and 
  $a_i$ denote the number of hyperplanes containing $i$ holes. Then, $\sum_{i=0}^{c}a_i = \gaussm{n}{1}{q}$, 
  $\sum_{i=0}^{c}ia_i = c\cdot \gaussm{n-1}{1}{q}$ and $\sum_{i=0}^{c}i(i-1) a_i = c(c-1)\cdot \gaussm{n-2}{1}{q}$.
\end{Lemma}
\begin{Proof}
  Double-count the incidences of the tuples $(H)$, $(B_1,H)$, and $(B_1,B_2,H)$, where $H$ is a hyperplane 
  and $B_1\neq B_2$ are points contained in $H$.
\end{Proof}

\begin{Lemma}
  \label{lemma_hyperplane_types_arithmetic_progression_c}
  Let $\Delta=q^{s-1}$, $m\in\mathbb{Z}$, and $\mathcal{P}$ be a vector space partition of $\F{n}{q}$ of hole-type $(t,s,c)$. Then, 
  $\tau_q(c,\Delta,m)\cdot \frac{q^{n-2}}{\Delta^2}-m(m-1)\ge 0$, 
  where $$\tau_q(c,\Delta,m)=m(m-1)\Delta^2q^2-c(2m-1)(q-1)\Delta q+c(q-1)\Big(c(q-1)+1\Big).$$
\end{Lemma}
\begin{Proof}
  Consider the three equations from Lemma~\ref{lemma_standard_equations_reduced}. $(c-m\Delta)\Big(c-(m-1)\Delta\Big)$ times the first minus 
  $\Big(2c-(2m-1)\Delta-1\Big)$ times the second plus the third equation, and then divided by $\Delta^2/(q-1)$, gives
  $$
    (q-1)\cdot\sum_{h=0}^{\left\lfloor c/\Delta\right\rfloor} (m-h)(m-h-1)a_{c-h\Delta}=\tau_q(c,\Delta,m)\cdot \frac{q^{n-2}}{\Delta^2}-m(m-1)
  $$ 
  due to Lemma~\ref{lemma_modulo}. Finally, we observe $a_i\ge 0$ and $(m-h)(m-h-1)\ge 0$ for all $m,h\in\mathbb{Z}$.
\end{Proof}

\begin{Lemma}
  \label{lemma_special_excluded_vsp}
  For integers $n>t\ge s\ge 2$ and $1\le i\le s-1$, there exists no vector space partition $\mathcal{P}$ of $\F{n}{q}$ of hole-type $(t,s,c)$, 
  where $c=i\cdot q^s-\gaussm{s}{1}{q}+s-1$.\footnote{For more general non-existence results of vector space partitions see e.g. 
  \cite[Theorem 1]{heden2009length} and the related literature.} 
\end{Lemma}
\begin{Proof}
  Assume the contrary and apply Lemma~\ref{lemma_hyperplane_types_arithmetic_progression_c} with $m=i(q-1)$. Setting $y=s-1-i$ 
  and $\Delta=q^{s-1}$ we compute
  \begin{eqnarray*}
    \tau_q(c,\Delta,m)&=&-q\Delta(y(q-1)+2) +(s-1)^2q^2 -q(s-1)(2s-5) +(s-2)(s-3).
  \end{eqnarray*}
  Using $y\ge 0$ we obtain $\tau_2(c,\Delta,m)\le s^2+s-2^{s+1}<0$. For $s=2$, we have $\tau_q(c,\Delta,m)= -q^2+q<0$ and 
  for $q,s>2$ we have $\tau_q(c,\Delta,m)\le -2q^s+(s-1)^2q^2<0$. Thus, Lemma~\ref{lemma_hyperplane_types_arithmetic_progression_c} 
  yields a contradiction. 
\end{Proof}

\begin{Theorem} (C.f.~\cite[Lemma 10]{nastase2016maximum}.)
  \label{main_theorem_1}
  For integers $r\ge 1$, $k\ge 2$, $u\ge 0$, and $0\le z\le \gaussm{r}{1}{q}/2$ with $t=\gaussm{r}{1}{q}+1-z+u>r$ we have
  $A_q(n,2t;t)\le lq^t+1+z(q-1)$, where $l=\frac{q^{n-t}-q^r}{q^t-1}$ and $n=kt+r$.   
\end{Theorem}
\begin{Proof}
  Apply Lemma~\ref{lemma_application_partial_spreads} with $x=2+z(q-1)$ and $y=z+1$. 
  If $z=0$, then $L<0$. For $z\ge 1$, apply Lemma~\ref{lemma_special_excluded_vsp}.
  Thus, $A_q(n,2t;t)\le lq^t+ x-1$.
\end{Proof}

The known constructions for partial $t$-spreads give $A_q(kt+r,2t;t)\ge lq^t+1$, see e.g.~\cite{beutelspacher1975partial} 
(or \cite{kurzspreads} for an interpretation using the more general multilevel construction for subspace codes). Thus, 
Theorem~\ref{main_theorem_1} is tight for $t\ge\gaussm{r}{1}{q}+1$, c.f.~\cite[Theorem 5]{nastase2016maximum}.

\newcommand{\uu}{\lambda}

\begin{Theorem} (C.f.~\cite[Theorem 6,7]{nastase2016maximumII}.)
  \label{main_theorem_2}
  For integers $r\ge 1$, $k\ge 2$, $y\ge \max\{r,2\}$, $z\ge 0$ with $\uu=q^{y}$, $y\le t$, $t=\gaussm{r}{1}{q}+1-z>r$, $n=kt+r$, 
  and  $l=\frac{q^{n-t}-q^r}{q^t-1}$, we have 
  $$
     A_q(n,2t;t)\le lq^t+\left\lceil \uu -\frac{1}{2}-\frac{1}{2}
    \sqrt{1+4\uu\left(\uu-(z+y-1)(q-1)-1\right)} \right\rceil.
  $$    
\end{Theorem}
\begin{Proof}
  From Lemma~\ref{lemma_application_partial_spreads} we conclude $L\le (z+y-1)q^y-(x-1)\gaussm{y}{1}{q}$ and $L\equiv -(x-1)\gaussm{y}{1}{q} 
  \pmod {q^y}$ for the number of holes of a certain $(n-t+y)$-dimensional subspace $U$ of $\mathbb{F}_q^n$. 
  $\mathcal{P}_U:=\{P\cap U\mid P\in\mathcal{P}\}$ is of hole-type $(t,y,L)$ if $y\ge 2$. Next, we will show that 
  $\tau_q(c,\Delta,m)\le 0$, where $\Delta=q^{y-1}$ and $c=iq^{y}-(x-1)\gaussm{y}{1}{q}$ with $1\le i\le z+y-1$, for suitable integers $x$ and $m$. 
  Note that, in order to apply Lemma~\ref{lemma_application_partial_spreads}, we have to satisfy $x\ge 2$ and $y\ge f$ for all integers $f$ with $q^f|x-1$.   
  Applying Lemma~\ref{lemma_hyperplane_types_arithmetic_progression_c} then gives the desired contradiction, so that 
  $A_q(n,2t;t)\le lq^t+x-1$.
  
  We choose\footnote{ 
  Solving $\frac{\partial \tau_q(c,\Delta,m)}{\partial m}=0$, i.e., minimizing $\tau_q(c,\Delta,m)$, yields $m=i(q-1)-(x-1)+\frac{1}{2}+\frac{x-1}{q^y}$. 
  For $y\ge r$ we can assume $x-1<q^y$ due 
  the known constructions for partial spreads,  
  so that up-rounding yields the optimum integer choice. For $y<r$ 
  the interval $\left[u+\frac{1}{2}- \frac{1}{2}\theta(i),u+\frac{1}{2}+ \frac{1}{2}\theta(i)\right]$ may contain no integer. 
  } $m=i(q-1)-(x-1)+1$, so that $\tau_q(c,\Delta,m)=x^2-(2\uu+1)x+
  \uu(i(q-1)+2)$. Solving $\tau_q(c,\Delta,m)=0$ for $x$ gives 
  $x_0=\uu+\frac{1}{2}\pm \frac{1}{2}\theta(i)$,  
  where $\theta(i)=\sqrt{1-4i\uu(q-1)+4\uu(\uu-1)}$. We have $\tau_q(c,\Delta,m)\le 0$ for 
  $\left|2x-2\uu-1\right|\le \theta(i)$. We need to find an integer $x\ge 2$ such that this inequality is satisfied for all
  $1\le i\le z+y-1$. The strongest restriction is attained for $i=z+y-1$. Since $z+y-1\le\gaussm{r}{1}{q}$ and $u=q^y\ge q^r$, we have 
  $\theta(i)\ge\theta(z+y-1)\ge 1$, so that $\tau_q(c,\Delta,m)\le 0$ for 
  $x= \left\lceil u+\frac{1}{2}- \frac{1}{2}\theta(z+y-1)\right\rceil$. (Observe $x\le \uu+\frac{1}{2}+ \frac{1}{2}\theta(z+y-1)$ due 
  to $\theta(z+y-1)\ge 1$.) Since $x\le \uu+1$, we have $x-1\le \uu=q^y$, so that $q^f | x-1$ implies $f\le y$ provided $x\ge 2$. The latter is true due to 
  $\theta(z+y-1)\le \sqrt{1-4\uu(q-1)+4\uu(\uu-1)}\le \sqrt{1+4\uu(\uu-2)}< 2(\uu-1)$, which implies $x\ge\left\lceil\frac{3}{2}\right\rceil=2$.
  
  So far we have constructed a suitable $m\in\mathbb{Z}$ such that $\tau_q(c,\Delta,m)\le 0$ for 
  $x=\left\lceil \uu+\frac{1}{2}- \frac{1}{2}\theta(z+y-1)\right\rceil$. If $\tau_q(c,\Delta,m)< 0$, then 
  Lemma~\ref{lemma_hyperplane_types_arithmetic_progression_c} gives a contradiction, so that we assume 
  $\tau_q(c,\Delta,m)=0$ in the following. If $i<z+y-1$ we have $\tau_q(c,\Delta,m)<0$ due to 
  $\theta(i)>\theta(z+y-1)$, so that we assume $i=z+y-1$. Thus, $\theta(z+y-1)\in\mathbb{N}_0$. 
  However, we can write  
  $
    \theta(z+y-1)^2=1+4\uu\left(\uu-(z+y-1)(q-1)-1\right)=(2w-1)^2 =1+4w(w-1)
  $
  for some integer $w$. If $w\notin\{0,1\}$, then $\gcd(w,w-1)=1$, so that either $\uu=q^y\mid w$ or $\uu=q^y\mid w-1$. Thus, in any case, 
  $w\ge q^y$, which is impossible since $(z+y-1)(q-1)\ge 1$. Finally, $w\in\{0,1\}$ implies $w(w-1)=0$, so that $\uu-(z+y-1)(q-1)-1=0$. Thus, 
  $z+y-1=\gaussm{y}{1}{q}\ge \gaussm{r}{1}{q}$ since $y\ge r$. The assumptions $y\le t$ and $t=\gaussm{r}{1}{q}+1-z$ imply
  $z+y-1=\gaussm{r}{1}{q}$ and $y=r$. This gives $t=r$, which is excluded.
\end{Proof}

Setting $y=t$ in Theorem~\ref{main_theorem_2} yields \cite[Corollary~8]{nets_and_spreads}, 
which is based on \cite[Theorem~1B]{bose1952orthogonal}. And indeed, our analysis is very similar to the technique\footnote{Actually, 
their analysis grounds on \cite{plackett1946design} and is strongly related to the classical second-order Bonferroni Inequality 
\cite{bonferroni1936teoria,galambos1977bonferroni,galambos1996bonferroni} in Probability Theory, see e.g.\ 
\cite[Section 2.5]{honold2015constructions} for another application for bounds on subspace codes.} used in 
\cite{bose1952orthogonal}. Compared to \cite{bose1952orthogonal,nets_and_spreads}, the new ingredients essentially are 
lemmas~\ref{lemma_modulo} and \ref{key_lemma}, see also \cite[Proof of Lemma 9]{nastase2016maximum}. 
\cite[Corollary~8]{nets_and_spreads}, e.g., gives $A_2(15,12;6)\le 516$, $A_2(17,14;7)\le 1028$, and $A_9(18,16;8)\le 3486784442$, 
while Theorem~\ref{main_theorem_2} gives $A_2(15,12;6)\le 515$, $A_2(17,14;7)\le 1026$, and $A_9(18,16;8)\le 3486784420$.
Postponing the details and proofs to a more extensive and technical paper \cite{costchapter}, we state: 
\begin{itemize}
  \item $2^4l+1\le A_2(4k+3,8;4)\le 2^4l+4$, where $l=\frac{2^{4k-1}-2^3}{2^4-1}$ and $k\ge 2$, e.g., 
        $A_2(11,8;4)\le 132$;
  \item $2^6l+1\le A_2(6k+4,12;6)\le 2^6l+8$, where $l=\frac{2^{6k-2}-2^4}{2^6-1}$ and $k\ge 2$, e.g., 
        $A_2(16,12;6)\le 1032$;
  \item $2^6l+1\le A_2(6k+5,12;6)\le 2^6l+18$, where $l=\frac{2^{6k-1}-2^5}{2^6-1}$ and $k\ge 2$, e.g., 
        $A_2(17,12;6)\le 2066$;      
  \item $3^4l+1\le A_3(4k+3,8;4)\le 3^4l+14$, where $l=\frac{3^{4k-1}-3^3}{3^4-1}$ and $k\ge 2$, e.g., 
        $A_3(11,8;4)\le 2201$; 
  \item $3^5l+1\le A_3(5k+3,10;5)\le 3^5l+13$, where $l=\frac{3^{5k-2}-3^5}{3^3-1}$ and $k\ge 2$, e.g., 
        $A_3(13,10;5)\le 6574$; 
  \item $3^5l+1\le A_3(5k+4,10;5)\le 3^5l+44$, where $l=\frac{3^{5k-1}-3^4}{3^5-1}$ and $k\ge 2$, e.g., 
        $A_3(14,10;5)\le 19727$; 
  \item $3^6l+1\le A_3(6k+4,12;6)\le 3^6l+41$, where $l=\frac{3^{6k-2}-3^4}{3^6-1}$ and $k\ge 2$, e.g., 
        $A_3(16,12;6)\le 59090$; 
  \item $3^6l+1\le A_3(6k+5,12;6)\le 3^6l+133$, where $l=\frac{3^{6k-1}-3^5}{3^6-1}$ and $k\ge 2$, e.g., 
        $A_3(17,12;6)\le 177280$; 
  \item $3^7l+1\le A_3(7k+4,14;7)\le 3^7l+40$, where $l=\frac{3^{7k-3}-3^4}{3^7-1}$ and $k\ge 2$, e.g., 
        $A_3(18,14;7)\le 177187$; 
  \item $4^5l+1\le A_4(5k+3,10;5)\le 4^5l+32$, where $l=\frac{4^{5k-2}-4^3}{4^5-1}$ and $k\ge 2$, e.g., 
        $A_4(13,10;5)\le 65568$; 
  \item $4^6l+1\le A_4(6k+3,12;6)\le 4^6l+30$, where $l=\frac{4^{6k-3}-4^3}{4^6-1}$ and $k\ge 2$, e.g., 
        $A_4(15,12;6)\le 262174$; 
  \item $4^6l+1\le A_4(6k+5,12;6)\le 4^6l+548$, where $l=\frac{4^{6k-1}-4^5}{4^6-1}$ and $k\ge 2$, e.g., 
        $A_4(17,12;6)\le 4194852$; 
  \item $4^7l+1\le A_4(7k+4,14;7)\le 4^7l+128$, where $l=\frac{4^{7k-3}-4^4}{4^7-1}$ and $k\ge 2$, e.g., 
        $A_4(18,14;7)\le 4194432$; 
  \item $5^5l+1\le A_5(5k+2,10;5)\le 5^5l+7$, where $l=\frac{5^{5k-3}-5^2}{5^5-1}$ and $k\ge 2$, e.g., 
        $A_5(12,10;5)\le 78132$; 
  \item $5^5l+1\le A_5(5k+4,10;5)\le 5^5l+329$, where $l=\frac{5^{5k-1}-5^4}{5^5-1}$ and $k\ge 2$, e.g., 
        $A_5(14,10;5)\le 1953454$; 
  \item $7^5l+1\le A_7(5k+4,10;5)\le 7^5l+1246$, where $l=\frac{7^{5k-1}-7^2}{7^5-1}$ and $k\ge 2$, e.g., 
        $A_7(14,10;5)\le 40354853$; 
  \item $8^4l+1\le A_8(4k+3,8;4)\le 8^4l+264$, where $l=\frac{8^{4k-1}-8^3}{8^4-1}$ and $k\ge 2$, e.g., 
        $A_8(11,8;4)\le 2097416$; 
  \item $8^5l+1\le A_8(5k+2,10;5)\le 8^5l+25$, where $l=\frac{8^{5k-3}-8^2}{8^5-1}$ and $k\ge 2$, e.g., 
        $A_8(12,10;5)\le 2097177$; 
  \item $8^6l+1\le A_8(6k+2,12;6)\le 8^6l+21$, where $l=\frac{8^{6k-4}-8^2}{8^6-1}$ and $k\ge 2$, e.g., 
        $A_8(14,12;6)\le 16777237$; 
  \item $9^3l+1\le A_9(3k+2,6;3)\le 9^3l+41$, where $l=\frac{9^{3k-1}-9^2}{9^3-1}$ and $k\ge 2$, e.g., 
        $A_9(8,6;3)\le 59090$; 
  \item $9^5l+1\le A_9(5k+3,10;5)\le 9^5l+365$, where $l=\frac{9^{5k-2}-9^3}{9^5-1}$ and $k\ge 2$, e.g., 
        $A_9(13,10;5)\le 43047086$; 
\end{itemize}
c.f.\ the web-page mentioned in footnote~1 for more numerical values and comparisons of the different upper 
bounds.

\end{document}